\documentclass[11pt]{amsart}
\usepackage{graphicx}
\usepackage{amsmath, amssymb}
\usepackage{color}
\newtheorem{thm}{Theorem}[section]
\newtheorem{cor}[thm]{Corollary}
\newtheorem{conj.}[thm]{Conjecture}
\newtheorem{lem}[thm]{Lemma}
\newtheorem{prop}[thm]{Proposition}
\theoremstyle{definition}
\newtheorem{defn}[thm]{Definition}
\theoremstyle{remark}
\newtheorem{rem}[thm]{Remark}
\numberwithin{equation}{section}

\newcommand{\norm}[1]{\left\Vert#1\right\Vert}

\newcommand{\RR}{\mathbb R}

\newcommand{\h}{\mathcal{H}}

\newcommand{\g}{\mathcal{G}}

\newcommand{\paff}{\pi_{\mathrm{aff}}}

\newcommand{\range}[1]{\operatorname{ran}\left( #1 \right)} 
\newcommand{\identity}[1]{\mathsf{id}_{ #1 }}
\newcommand{\kernel}[1]{\ker\left( #1 \right)}

\definecolor{darkviolet}{rgb}{0.58,0,0.83} 

\begin{document}

\title[Continuous frames in tensor product Hilbert spaces]{Continuous frames in tensor product Hilbert spaces, localization operators and density
operators}%
\author[P. Balazs, N. Teofanov]{P. Balazs$^1$, N. Teofanov$^{2}$,}
\address{$^1$ Acoustics Research Institute, Austrian Academy of Sciences, Wohllebengasse 12-14, 1040 Wien, Austria.}
\email{peter.balazs@oeaw.ac.at}
\address{$^{2}$ Department of Mathematics and Informatics, University of Novi Sad Faculty of Sciences, Trg D. Obradovi\'ca 4, Novi Sad, Serbia.}
\email{nenad.teofanov@dmi.uns.ac.rs}

\subjclass[2000]{Primary 42C15; Secondary 42C40, 47A58}
\keywords{Continuous frames, Dual frames, Tensor product, Bilinear continuous frame multipliers, Localization operators, Patrial traces, Density operators and matrices, Quantum systems.}

\begin{abstract}
Continuous frames and tensor products are important topics in theoretical physics. This paper combines those concepts.
We derive fundamental properties of continuous frames for tensor product of Hilbert spaces.
This includes, for example, the consistency property,  i.e. preservation of the frame property under the tensor product,
and the description of the canonical dual tensors by those on the Hilbert space level.
We show the full characterization of all dual systems for a given continuous frame, a result  interesting by itself, and
apply this to dual tensor frames. Furthermore, we discuss the existence on non-simple tensor product (dual) frames.
Continuous frame multipliers and their  Schatten class properties are considered in the context of tensor products.
In particular, we give sufficient conditions for obtaining partial trace multipliers of the same form, which is illustrated with examples related to short-time Fourier transform and wavelet localization operators. As an application, we offer an interpretation of a class of tensor product continuous frame multipliers as density operators for bipartite quantum states, and show how their structure can be restricted to the corresponding partial traces.
\end{abstract}

\maketitle

\section{Introduction}

Continuous frames extend the concept of frames when the indices are related to some measurable space, see \cite{Ali1, Ra, Gb, Gk}.
Apart from expected similarities, this extension pointed out various differences between the ``discrete`` and ``continuous`` theories.
For example, continuous frames need not be norm bounded, and they may describe the states of quantum systems in a neighborhood
of a point in phase space $\mathbb{R}^{2d} $, which is a more realistic situation than the corresponding discrete case related
to some  lattice in $\mathbb{R}^{2d} $, cf. \cite{deGosson2020}.

On the other hand the tensor product of Hilbert spaces is a very important topic in mathematics \cite{LOAN200085} and theoretical physics \cite{Caban_2005}. Here we combine those two approaches.
We introduce the notion of continuous frames (and Bessel mappings) for tensor products of Hilbert spaces $\h = \h_1 \otimes \h_2 $ with respect to a
(tensor product) measure space $(X,\mu)$. When the measure $\mu $ is chosen to be the counting measure, the main properties of tensor products of
(discrete) frames considered in \cite{bourou, Garcia, KhoAsg, WangLi} are recovered.

We show the expected consistence property, i.e. that the continuous frame/Bessel mapping  condition is preserved by the tensor product, Theorem
\ref{thm:main}.  To tackle the issue of representing vectors in tensor product Hilbert spaces, different systems can be used for analysis and
synthesis, which gives rise to the notion of dual pairs of continuous frames. We study the corresponding operators,
and give a representation of canonical dual frames for the tensor product continuous frames. 
In addition, we briefly discuss the existence of non-simple tensor product (dual) frames. For that result, a full characterization of all dual continuous frames is needed. We prove the generalization of the well-known result for (discrete) frames \cite[Lemma 6.3.6]{ole} to the continuous frame setting, solving an open question. We use the powerful technique of reproducing kernel Hilbert spaces in these investigations, see Theorem \ref{olecontinuous}, and derive the corresponding property for tensor product continuous frames.

\par

Let us recall that tensor product Hilbert spaces are important in many different contexts. For example, as noted in \cite{xxlgrospeck19}, {\em ``the theory of tensor products is at the heart
of kernel theorems for operators``}. In fact, tensor product of two Hilbert spaces can be introduced in terms of Hilbert-Schmidt operators which, in turn, can be identified with their kernels.  In this paper we focus our attention to other aspects of tensor products, and the
approach based on kernel theorems will be given in a separate contribution.

\par

For example, in Section \ref{sec:multipliers} we study the tensor product continuous frame multipliers and their compactness properties,
thus extending results from \cite{BBR} to the tensor product setting. In addition,
we recall the partial trace theorem which is an important tool related to applications of our results to quantum systems.
As an illustration, in  Section \ref{sec:STFTwavmult0} we consider particular examples of continuous frame multipliers
in the form of familiar localization operators in the context of the short-time Fourier transform and wavelet multipliers.
Localization operators are used in the context of quantization \cite{Berezin71}, in signal analysis \cite{da88},
or as an approximation of pseudodifferential operators, cf. \cite{CRodino} and the references given there.
We recover some well-known results, but also point out some Schatten class results related to the wavelet and mixed type multipliers that so far seems to remain unconsidered.

\par

Specific instances of our general theory, which is one of the main motivations for our study
could be related to the states of quantum systems. More precisely, we
propose the interpretation of a family of trace class operators as density operators also called density matrices
for composite (bipartite) quantum systems. Recently, de Gosson in \cite{deGosson2020, deGosson2021} considered Toeplitz density operators by using the approach which is closely related
to the short-time Fourier transform multipliers of Section \ref{sec:STFTwavmult0}.
The main feature of operators considered in Section \ref{sec:quantum} is that their partial traces
(or reduced density operators) are operators of the same form.
Thus we propose the study of  bilinear localization
operators which, in principle, could be used to describe the state of subsystem in a prescribed region of the phase space.
This is analogous to the use of localization operators in extracting an information about a signal
in a specific region of time-frequency plane.

In our opinion the results from Sections  \ref{sec:STFTwavmult0} and  \ref{sec:quantum},
open the perspective of using mathematical tools developed in Sections \ref{sec:contframes} and
\ref{sec:multipliers} in the future study of bipartite quantum systems and their subsystems.
For example, Theorem \ref{thm:density} provides a description of  the separable state of a composite system, and a
partial affirmative answer to the question of  de Gosson  \cite[Section 5]{deGosson2020} which can be roughly rephrased as follows:
can the structure of a density operator be appropriately restricted to its partial traces?

\par

\section{Preliminaries} \label{sec:preliminaries}

For the reader's convenience in this section we collect some basic facts from operator theory  and tensor products of Hilbert spaces  which will be
used in the sequel. We refer to \cite{conw1, Fol, 
 Mu} for details.

\subsection{Operator theory}

By $\h $ we denote a complex Hilbert space with the inner product $\langle x, y\rangle$
(linear in the first and conjugate linear in the second coordinate)
and norm $\|x\| = \sqrt{\langle x, x  \rangle}$, $x,y \in \h$. In the sequel we consider separable Hilbert spaces.
A map $\Psi:\h \times \h\rightarrow \mathbb{C}$ is a sesquilinear form if it is linear in
the first variable and conjugate-linear in the second. The sesquilinear form
is bounded if there exists a constant $C>0$ such that $\left| \Psi(x,y) \right| \le C \cdot \| x \| \| y \|$,
$x,y \in \h$. The smallest, optimal, such constant is called the bound of $\Psi$ denoted by $\|\Psi\|$. There is a unique operator $O$ on $\h$ such that
\begin{equation} \label{murphy}
\Psi(x,y)=\langle O(x),y\rangle \quad x,y\in \h,
\end{equation}
and $\|O\|=\|\Psi\|$.

\par

A bounded operator $T: \h \rightarrow \h$ is positive (respectively
non-negative), if $\langle Tx,x\rangle>0$ for all $x\neq0$ (respectively $\langle Tx,x\rangle\geq0$ for all $x\in\h$).


A linear operator $T$ from the Banach space $X$ into the Banach space $Y$ is
compact if the image 
 of the closed unit ball in $X$ is a relatively compact subset of $Y$, or, equivalently, if the image of any bounded
sequence contains a convergent subsequence.
If  $T$ is a compact operator on Hilbert space $\h$ and if $T^{\ast}$ is the adjoint of $T$ (i.e. $\langle Tx, y \rangle = \langle x, T^{\ast}y \rangle$, $ \forall x,y \in \h$) then  the eigenvalues of
the unique non-negative and compact operator $S$ such that $S^2 = T^{\ast}T$  are called the
singular values of $T$.
An operator  $T $ belongs to the Schatten class $\mathcal{S}_p(\h)$,  $1 \leq p < \infty$,
if the sequence of its singular values $(s_n)$ belongs to $\ell^p$. In particular,
$\mathcal{S}_1(\h)$ consists of the trace class operators, and  $\mathcal{S}_2(\h)$ is the class of Hilbert-Schmidt operators (see also below),
and $\mathcal{S}_p(\h) \subset \mathcal{S}_q(\h)$, when $ 1\leq p \leq q \leq \infty,$ where $ S_\infty (\h) = \mathcal{B}(\h) $  denotes the set of all bounded
linear operators on $\h$.

\par

Let $\h_1$ and $\h_2$ be separable Hilbert spaces.
The set $\mathcal{B}(\h_2,\h_1)$ of all bounded linear operators from $\h_2$ to $\h_1$
is a Banach space with the usual operator norm $\|T\|=\sup_{\|x\|=1}\|Tx\|$, and $GL(\h_2,\h_1)$ denotes  the set
of all bounded linear operators  from $\h_2$ to $\h_1$ with bounded inverse. If $\h_1 = \h_2 = \h$, we  write $\mathcal{B}(\h)$ and $GL(\h)$ for
short.

If $ T \in \mathcal{B}(\h_2,\h_1)$ and
$$
\| T \| ^2 _{\h S} := \sum_{n=1} ^\infty \| T e_n \|^2 _{\h_1} < \infty
$$
for some orthonormal basis (ONB)  $ (e_n ) $ in $ \h_2 $, then $T$ is called a {\em Hilbert-Schmidt} ($ \h S $) operator from $\h_2$ to $\h_1$.
We denote the class of  Hilbert-Schmidt  operators by  $\h S (\h_2,\h_1 )$.
If $\h_1 = \h_2 = \h$, then $\h S (\h,\h ) = \mathcal{S}_2(\h)$.
$\h S (\h_2,\h_1 )$ is a Hilbert space (of compact operators) with the inner product
$$
\langle S, T \rangle_{\h S} = \sum_{n=1} ^\infty  \langle S e_n, T e_n \rangle_{\h_1}.
$$

If  $ x\in \h_1$ and $y \in \h_2 $, then their {\em tensor product} $x \otimes y : \h_2 \rightarrow \h_1$ is defined by
\begin{equation} \label{outten1}
(x \otimes y ) h = \langle h, y \rangle x, \;\;\; h \in \h_2,
\end{equation}
belonging to $ \h S (\h_2,\h_1 )$.

\par

For $ P \in \mathcal{B}(\h_2)$ and $ Q \in \mathcal{B}(\h_1)$ we define the {\em tensor product of operators} $ Q \otimes P :
\mathcal{B}(\h_2, \h_1) \rightarrow \mathcal{B}(\h_2, \h_1) $ by
$ (Q \otimes P) T = Q \circ T \circ P^*.$ It is invertible if and only if $P$ and $Q $ are invertible, and
$ (Q \otimes P)^{-1} =  Q^{-1} \otimes P^{-1} .$

\subsection{Tensor product of Hilbert spaces}

Let $\h_1$ and $\h_2$ be separable Hilbert spaces.
Equipping the algebraic tensor product with the (extension of) the inner product
\begin{equation}
\langle x_1 \otimes y_1, x_2 \otimes y_2 \rangle_{\otimes} =
\langle x_1, x_2 \rangle_{\h_1} \langle  y_1,  y_2 \rangle_{\h_2}, \;\;\; x_1, x_2 \in \h_1,\; y_1, y_2 \in \h_2, \label{eq:tens innprod1}
\end{equation}
makes it into a Hilbert space, denoted by $\h_1 \otimes \h_2$, and
$ \| \cdot \|_{\otimes} = \langle \cdot, \cdot \rangle_{\otimes}.$

The space  $\h_1 \otimes \h_2$ is unitary isomorphic to the class of Hilbert-Schmidt operators  $\h S (\h_2,\h_1 )$.
The unitary operator maps $( x_1 \otimes y_1)$ onto the operator given by \eqref{outten1}
cf. \cite{Heil}.

\par

Let us collect basic properties of tensor products given in the following lemma.

\begin{lem}  \label{lm:tensorprodprop}
Let  $\h_1$ and $\h_2$ be separable Hilbert spaces and $\h_1 \otimes \h_2$ their tensor product. Then we have:
\begin{itemize}
\item[a)] $ \| u \otimes v \| = \| u \| \| v\|,$ $u \in \h_1,$  $v \in \h_2.$
\item[b)] if $ S \in \mathcal{B} (\h_1) $ and $ T \in \mathcal{B} (\h_2), $ then
$ \| S \otimes T \| = \|   S \| \| T\|,$ and
$$
(S \otimes T) (u \otimes v) = Su \otimes Tv, \qquad u \in \h_1,v \in \h_2.
$$
\item[c)] $\h_1 \otimes \h_2 = \overline{\text{span}} \; \{  u \otimes v, \;:\;  u \in \h_1, v \in \h_2 \}$,
i.e. $\h_1 \otimes \h_2$ is the
    closure of the set of all
finite linear combinations of elements of the form $ u \otimes v$, $u \in \h_1$, $v \in \h_2$.
\item[d)] the tensor product of two ONBs is an ONB in the tensor product space.
\item[e)] (Schmidt decomposition) for every $ x \in \h_1 \otimes \h_2$ there are non-negative numbers $ c_n$ and ONB $ e_n \in \h_1$  and
$ f_n \in \h_2,$  such that
$$
 x_ = \sum_{n=1} ^\infty c_n (e_n \otimes f_n), \qquad \| x\| = \sum_{n=1} ^\infty c_n ^2.
$$
\end{itemize}
\end{lem}

\begin{proof} The proof a)--d) is folklore, see e.g. \cite{Fol,gaal,Hall}. For the proof of Schmidt decomposition e) we refer to \cite{BB}.
\end{proof}



\section{Frames in tensor products of Hilbert spaces} \label{sec:contframes}

In this section we derive fundamental properties of continuous frames for tensor product of Hilbert spaces.
%
%
%
%

The usual definition of frames use discrete index sets \cite{ole}, one can also give a definiton using continuous ones \cite{Ali1, Ali2, BBR}. As an introduction we transfer the basic definitions directly for the case used in this manuscript.

\subsection{Continuous Frames in Tensor Product Hilbert spaces}

\begin{defn} \label{De:frame-Bessel}
Let $\h$ be the tensor product $\h = \h_1 \otimes \h_2 $ of separable complex Hilbert spaces, and $(X, \mu) = ( X_1 \times X_2, \mu_1 \otimes
\mu_2)$ be the product of measure spaces with $\sigma-$finite positive measures $\mu_1, \mu_2.$
The mapping $F :X\to\h$ is called a \emph{continuous frame for the tensor product Hilbert space} $\h$ with respect to $(X ,\mu)$, if
\begin{enumerate}
\item  $F$ is weakly-measurable, i.e., for all $\vec{f} \in \h$,
$$
x = (x_1, x_2) \to
\langle \vec{f}, F({x})\rangle
$$
is a measurable function on $X$; \\
\item  there exist constants  $A>0$ and $B<\infty$ such that
\begin{equation}
\label{deframe}
A\|\vec{f} \|^{2}  \leq \int_{X}|\langle \vec{f},F({x})\rangle|^{2}\,d\mu(x)
\leq B\|\vec{f}\|^{2}, \qquad \forall \vec{f}\in \h.
\end{equation}
\end{enumerate}

The constants $A$ and $B$ are called the lower and the upper \emph{continuous frame bound}, respectively.
If $A=B$, then $F$ is called a \emph{tight} continuous frame, if $A = B = 1$ a {\em Parseval} frame.

The mapping $F$ is called the {\em Bessel mapping} if only the second inequality in
(\ref{deframe}) is considered. In this case, $B$ is called the \emph{Bessel constant} or the \emph{Bessel bound}.\footnote{Please be aware, that this concept of Bessel mappings does not coincide with Bessel functions. }
\end{defn}

To each continuous frame we define the frame related operators as follows.

Let $(X, \mu) $ and $\h$ be as in Definition \ref{De:frame-Bessel}, and let
$ L^{2}(X, \mu)$ be the space of square-integrable functions on
$(X, \mu) $. The operator $T_{F}:L^{2}(X, \mu)\to\h$ defined by
\begin{equation} \label{eq:synthesisopr}
T_{F}\vec{\varphi} = \int_{X} \vec{\varphi}(x) F(x) \,d\mu(x) =  \int_{X_1} \int_{X_2} \vec{\varphi}(x_1, x_2) F(x_1, x_2)
\,d\mu_1 (x_1) \,d\mu_2 (x_2)
\end{equation}
is called the \textit{synthesis operator}, and the operator
$ T_{F}^{*}: \h \to L^{2}(X, \mu),$  given by
\begin{equation} \label{eq:adjointsynthesisopr}
 (T_{F}^{*}\vec{f})(x)=\langle \vec{f}, F(x)\rangle,\quad  x\in X
\end{equation}
is called the \textit{analysis operator} of $F$.

The \textit{continuous frame operator}  $S_F$  of $F$ is given by $S_{F}=T_{F}T_{F}^{*}$.


\begin{rem} \label{rem:Riesz bases}
In discrete frame theory, it is of interest to consider Riesz bases.
It does not make sense to address this question in the
context of continuous frames, since all continuous Riesz bases are actually discrete, cf. \cite{spexxl16,Ra,JAKOBSEN2016229}.
\end{rem}


The first inequality in (\ref{deframe}), shows that $F$ is complete, i.e.,$$\overline{\textrm{span}}\{F({x})\}_{x\in X}=\h,$$
where we have
$ \displaystyle
\overline{\textrm{span}}\{F({x})\}_{x\in X} := \{ f \in \h \;\; | \;\; \mu  \left( \left\{ x \left|  \langle f, F(x) \rangle \not= 0 \right\} \right. \right)  \neq 0 \}
$. In contrast to discrete setting, in the continuous setting one has to be a bit more careful with
this definition due to the null sets in $X$, cf. \cite{xxlosg21}.

%
%

The next result shows that the continuous frame condition is preserved by the tensor product, generalizing the result for discrete frames.

\begin{thm} \label{thm:main}
Let $\h_1$ and $\h_2$ be separable Hilbert spaces,  $\h = \h_1 \otimes \h_2 $, and let
$(X, \mu) = ( X_1 \times X_2, \mu_1 \otimes \mu_2)$ be the product of measure spaces with $\sigma-$finite positive measures $\mu_1, \mu_2.$
The mapping $F = F_1 \otimes F_2 :X\to\h$ is  a continuous frame for $\h$ with respect to $(X ,\mu)$ if and only if
$ F_1 $ is a continuous  frame for $\h_1$ with respect to $(X_1, \mu_1) $, and
$F_2 $ is a continuous  frame for $\h_2 $ with respect to   $(X_2, \mu_2) $.

Furthermore, if  $F = F_1 \otimes F_2 $ is  a continuous frame for $\h$ with frame bounds $A$ and $B$,
then the continuous frame bounds for $F_1$ can be chosen as
$ A_1 = A/  C_{F_2} $ and $ B_1 = B/  D_{F_2} $, where
\begin{equation} \label{eq:lowerboundF1}
C_{F_2} = \inf_{ \| g\|_{\h_2} = 1}
\int_{X_2} |\langle g , F_2(x_2) \rangle|^{2}\,  d\mu_2 (x_2),
\end{equation}
\begin{equation}  \label{eq:upperboundF1}
D_{F_2} = \sup_{ \| g\|_{\h_2} = 1}
\int_{X_2} |\langle g , F_2(x_2) \rangle|^{2}\,  d\mu_2 (x_2),
\end{equation}
and the continuous frame bounds for $F_2$ can be chosen as
$ A_2 = A/ C_{F_1}$ and $ B_2 = B/ D_{F_1}$, where
\begin{equation} \label{eq:lowerboundF2}
C_{F_1} =  \inf_{ \| f\|_{\h_1} = 1} \int_{X_1} |\langle f , F_1(x_1) \rangle|^{2}\,  d\mu_1 (x_1) ,
\end{equation}
\begin{equation}  \label{eq:upperboundF2}
D_{F_1} = \sup_{ \| f\|_{\h_1} = 1}
\int_{X_1} |\langle f , F_1(x_1) \rangle|^{2}\,  d\mu_1 (x_1).
\end{equation}

Vice versa, if $ F_j $ is a continuous  frame for $\h_j$ with the frame bounds $A_j$ and $B_j$, $ j = 1,2$,
then the frame bounds for $ F = F_1 \otimes F_2 $ can be chosen as $ A= A_1 A_2 $ and $B= B_1 B_2$.
\end{thm}

\begin{proof}
Assume that $F = F_1 \otimes F_2$ is a continuous  frame for $\h = \h_1 \otimes \h_2 $ with respect to $(X ,\mu)$. Let $f\in \h_1 \setminus \{ 0
\},$
and fix  $g\in \h_2 \setminus \{ 0 \}$. Then $ f\otimes g \in \h $, and
$$
T^* _{ F_1 \otimes F_2} ( f\otimes g) := \langle f\otimes g , F_1(x_1) \otimes F_2(x_2) \rangle =
\langle f , F_1(x_1) \rangle  \langle  g ,  F_2(x_2) \rangle 
$$
implies that (by Fubini's theorem)
\begin{multline*}
\int_X |\langle f\otimes g , F_1(x_1) \otimes F_2(x_2)  \rangle|^{2}\,d\mu(x) \\
= \int_{X_1} |\langle f, F_1(x_1) \rangle|^{2}\,d\mu_1(x_1)
\int_{X_2} |\langle g , F_2(x_2) \rangle|^{2}\,  d\mu_2(x_2).
\end{multline*}
Now, \eqref{deframe} and
$$
\| f\otimes g \|_{\otimes} = \| f\|_{\h_1 } \| g \|_{ \h_2}
$$
imply
$$
A\|f\otimes g \|_{\otimes} ^{2}   \leq \int_{X_1} |\langle f, F_1(x_1) \rangle|^{2}\,d\mu_1(x_1)
\int_{X_2} |\langle g , F_2(x_2) \rangle|^{2}\,  d\mu_2(x_2)
\leq B\| f\otimes g \|_{\otimes} ^{2},
$$
so that
\begin{align*}
\frac{A \| g \|^2 _{ \h_2}}{\int_{X_2} |\langle g , F_2(x_2) \rangle|^{2}\,  d\mu_2(x_2)  }
\| f\|^2 _{\h_1 } & \leq \int_{X_1} |\langle f, F_1(x_1) \rangle|^{2}\,d\mu_1(x_1)
\\
& \leq \frac{B \| g \|^2 _{ \h_2}}{\int_{X_2} |\langle g , F_2(x_2) \rangle|^{2}\,  d\mu_2(x_2)  }
\| f\|^2 _{\h_1 }.
\end{align*}
Notice that $\int_{X_2} |\langle g , F_2(x_2) \rangle|^{2}\,  d\mu_2(x_2) \neq 0$ for all $g \in \h_2 \setminus \left\{ 0 \right\}$, and choose
\begin{align*}
 A_1  & := \sup_{\| g \|^2 _{ \h_2} = 1}
 \frac{A }{\int_{X_2} |\langle g , F_2(x_2) \rangle|^{2}\,  d\mu_2(x_2)  } = \frac{A}{C_{F_2}} > 0, \\
 B_1  & :=  \inf_{ \| g \|^2 _{ \h_2} = 1}
 \frac{B}{\int_{X_2} |\langle g , F_2(x_2) \rangle|^{2}\,  d\mu_2(x_2)  } = \frac{B}{D_{F_2}} < \infty,
\end{align*}
with $C_{F_2}$ and $D_{F_2}$ given by \eqref{eq:lowerboundF1} and \eqref{eq:upperboundF1} respectively.

Thus we conclude that $ F_1 $ is a continuous frame for $ \h_1 $ with respect to $(X_1, \mu_1) $ with the continuous frame bounds $A_1$ and $B_2$.

By similar arguments we conclude that  $F_2 $ is a continuous frame for $ \h_2 $ with respect to $(X_2, \mu_2) $ with continuous frame bounds $  0< A_2   = A/C_{F_1}$ and $  B_2   = B/D_{F_1} < \infty,$
with $C_{F_1}$ and $D_{F_1}$ given by \eqref{eq:lowerboundF2} and \eqref{eq:upperboundF2} respectively.

\par

For the converse, by the assumptions it immediately follows that
$ F = F_1 \otimes F_2 $ is weakly measurable on $ \h $ with respect to $ (X, \mu) $, so it remains to check \eqref{deframe}.

Let $f\otimes g$ be a simple tensor. Then
\begin{multline*}
\norm{T^*  _{F_1 \otimes F_2} (f \otimes g)}^2 = \int_{X_1} \int_{X_2} \left| \left< f \otimes g , F_1 (x_1) \otimes F_2 (x_2) \right> \right|^2 d
\mu(x_1) d \mu(x_2) \\
 = \int \left| \left< f  , F(x_1)\right> \right|^2 d \mu(x_1) \int \left| \left< g , F_2 (x_2) \right> \right|^2 d \mu(x_2) \\
\le B_1 B_2 \norm{f}_{\h_1} ^2 \norm{g}_{\h_2} ^2 =  B_1 B_2 \norm{f \otimes g}_{\otimes} ^2,
 \end{multline*}
and similarly
$$
\norm{T^*  _{F_1 \otimes F_2} (f \otimes g)}^2  \geq A_1 A_2 \norm{f \otimes g}_{\otimes} ^2.
$$

This is true for the span of $f \otimes g$ which is dense in $\h_1 \otimes \h_2$.
By \cite[Proposition 2.5]{RaNaDe} it follows that $F_1 \otimes F_2$ is a continuous frame with the frame bounds $ A = A_1 A_2 $ and $B= B_1 B_2$.

\end{proof}

From the proof of Theorem \ref{thm:main} we also have the following observation.

\begin{cor}
Let the assumptions of Theorem \ref{thm:main} hold.
Then the  mapping $F = F_1 \otimes F_2 :X\to\h$ is  a continuous bilinear Bessel mapping for $\h$ with respect to $(X ,\mu)$ if and only if
$ F_1 $ is a continuous  Bessel mapping  for $\h_1$ with respect to $(X_1, \mu_1) $ and
$F_2 $ is a continuous Bessel mapping  for $\h_2 $ with respect to   $(X_2, \mu_2) $.
\end{cor}

\par

\subsection{Dual pairs of continuous frames}\label{sec:dualpair0}

Next we discuss dual continuous frames. If $F_j $ are continuous frames for $\h_j $, $j=1,2$,
then we may consider  dual frames $G_j$ which fulfill
$$
\langle f, g \rangle = \int_{X_j} \langle f, F_j (x_j) \rangle \langle G_j (x_j), g \rangle d\mu (x_j), \;\; \forall f,g \in \h_j, \; j=1,2,
$$
cf. Definition  \ref{def:dualframes} for a more general situation
(see also \cite{Gb}). It follows from Theorem \ref{thm:main} that such dual frames
give rise to continuous (dual) frames for the tensor product Hilbert space $ \h = \h_1 \otimes \h_2$.

In this subsection we focus on the frame operator in the context of tensor products,
and show that it gives rise to the {\em canonical dual frame} for a given  frame.
However, as we shall see, there always exist non-simple dual frames for tensor products of Hilbert spaces.

%
%

Let us note that $\langle S_{F}\vec{f},\vec{f}\rangle=\int_{X}|\langle \vec{f},F(x)\rangle |^{2}\,d\mu(x)$. Therefore \cite{Ali2,RaNaDe},
it follows that $0 <AI\leq S_{F}\leq BI$. Hence $S_{F}$ is invertible, positive and $1/B I\leq S^{-1}_{F}\leq 1/A I$.
Every $\vec{f}\in\h$ has (weak) representations of the form
\begin{eqnarray} \label{eq:dualframerepr}
\vec{f} & =  S_{F}^{-1}S_{F}\vec{f} =\int_{X}\langle \vec{f}, F(x)\rangle
S_{F}^{-1}F(x)\,d\mu(x) \\
& =S_{F}S_{F}^{-1} \vec{f}=\int_{X}\langle
\vec{f}, S_{F}^{-1}F(x)\rangle F(x)\,d\mu(x). \nonumber
\end{eqnarray}

\par

It can be proved that the mapping $F:X\to\h$ is a continuous
frame with respect to $(X, \mu)$ for $\h$ if and only if the
operator $S_{F}$ is a bounded and invertible operator.
(If $F$ is a Bessel mapping, then  $S_{F}$ is  bounded, selfadjoint and non-negative.)

To each continuous frame $F$ one can associate a dual continuous frame which is introduced as follows.

\begin{defn} \label{def:dualframes}
Let $F$ and $G$ be continuous frames for $\h = \h_1 \otimes \h_2$ with respect to $(X,\mu)= ( X_1 \times X_2, \mu_1 \otimes \mu_2)$.
The frame  $G$ is a continuous dual frame of $F$ if
$$
\vec{f} = \int_{X} \langle \vec{f},F(x)\rangle  G(x) d\mu(x), \qquad \forall \vec{f} \in\h,
$$
in the weak sense, i.e. if
\begin{equation}\label{dual} \langle \vec{f},\vec{g}\rangle=\int_{X}\langle
\vec{f},F(x)\rangle\langle G(x),\vec{g}\rangle d\mu(x), \quad   \forall  \vec{f}, \vec{g} \in\h.
\end{equation}

In this case the pair $(F,G)$ is called a \textit{dual pair of continuous frames}.
\end{defn}

\par

By Definition \ref{def:dualframes} and \eqref{eq:dualframerepr} it follows that for a given continuous frame $F$
there always exists an associated dual pair, i.e. $ (F, S_F ^{-1}  F)$ and  $ (S_F ^{-1}  F, F)$ are dual pairs.
The frame $S_F ^{-1}  F$ is called  {\em the canonical dual frame} for $F$, denoted by $\widetilde F (x)$.

\par

In the next theorem we establish the tensor product version of the usual identification of continuous frame operator in terms of analysis and
synthesis operators.
We refer to \cite{peter2,RaNaDe} when $\h$ is a Hilbert space.

\begin{thm}\label{TF}
Let $(X, \mu) = ( X_1 \times X_2, \mu_1 \otimes \mu_2)$  be a tensor product measure space and let $F$ be a Bessel
mapping from $X$ to $\h = \h_1 \otimes \h_2.$ Then the synthesis operator
$T_{F}:L^{2}(X, \mu)\to\h$ given  by \eqref{eq:synthesisopr}
is a well defined, linear and bounded operator, and its adjoint  operator
$ T_{F}^{*}: \h \to L^{2}(X, \mu)$ is given by \eqref{eq:adjointsynthesisopr}.

If $F = F_1 \otimes F_2 $ is  a continuous  frame for $\h$ with respect to $(X ,\mu)$,
and $\vec{f} = f_1 \otimes f_2 \in \h,$ then the analysis operator can be  represented by
\begin{equation} \label{eq:analysisoprepr}
(T_{F}^{*}\vec{f})(x)=\langle f_1, F_1 (x_1) \rangle \langle f_2, F_2 (x_2) \rangle
\end{equation}

The continuous frame operator $S_F$ is given by $S_{F}=T_{F}T_{F}^{*},$
and
$$
S_{F_1 \otimes F_2} =   S_{F_1} \otimes S_{F_2}.
$$
The canonical dual frame for $F$ is $ G = S_{F_1 } ^{-1} F_1 \otimes S _{F_2 } ^{-1} F_2$.
\end{thm}
\begin{proof} The first part of the claim follows immediately from the definition of $T_F$ given by \eqref{eq:synthesisopr} .
Furthermore, if $F = F_1 \otimes F_2 $ is  a continuous bilinear frame for $\h$ with respect to $(X ,\mu)$, then the representation
\eqref{eq:analysisoprepr} follows directly from \eqref{eq:tens innprod1}.

\par

It remains to show the second part of Theorem  \ref{TF}.
Let  $ f_j \in \h_j $, $ j=1,2.$ Then
\begin{align*}
T_F T_F^*(f_1\otimes f_2) & = T_F \left( \left< f_1, F_1 (x_1) \right> \left<f_2, F_2 (x_2)\right> \right) \\
& = \int_X \left< f_1, F_1 (x_1) \right> \left<f_2, F_2 (x_2)\right> F_1(x_1) \otimes F_2 (x_2) d \mu (x) \\
& = \int_{X_1} \left< f_1, F_1 (x_1) \right> F_1(x_1) d \mu_1 (x_1) \otimes \int_{X_2} \left<f_2, F_2 (x_2)\right> F_2 (x_2) d\mu_2 (x_2) \\
& = S_{F_1} f_1 \otimes S_{F_2} f_2 = \left( S_{F_1} \otimes S_{F_2} \right) \left(f_1 \otimes f_2 \right),
\end{align*}
and
\begin{align*}
T_F T_F^*(f_1\otimes f_2) & =
 \int_X \left< f_1, F_1 (x_1) \right> \left<f_2, F_2 (x_2)\right> F_1(x_1) \otimes F_2 (x_2) d \mu (x) \\
& = \int_{X} \left< f_1 \otimes f_2, F_1 (x_1) \otimes F_2 (x_2) \right> F_1(x_1) \otimes F_2 (x_2) d \mu (x)   \\
& = S_{F_1 \otimes F_2} (f_1 \otimes f_2 ).
\end{align*}

Therefore on simple tensors we have that $S_F = S_{F_1} \otimes S_{F_2}$. By Lemma
\ref{lm:tensorprodprop} (see also \cite[Proposition 2.5]{RaNaDe}) this is true on all of $\h$.

Moreover, $ S_F $ is self-adjoint and we have
$$
S^{-1} _F = (S_{F_1 } \otimes S_{F_2})^{-1} = S_{F_1 } ^{-1} \otimes S_{F_2} ^{-1}
= S_{G_1 }  \otimes S_{G_2}  = S_{G_1 \otimes G_2},
$$
where $ G_1 $ and $ G_2 $ are canonical dual frames of $F_1 $ and $ F_2 $ respectively.

Furthermore,
$$
S^{-1} _{F_1 \otimes F_2} ( F_1 \otimes F_2 ) =
S_{F_1 } ^{-1} \otimes S_{F_2} ^{-1} (F_1 \otimes F_2) =
(S_{F_1 } ^{-1} F_1) \otimes (S_{F_2} ^{-1} F_2) = G_1 \otimes G_2,
$$
which proves the claim.
\end{proof}

Recall that in $ L^{2} ( X_1 \times X_2, \mu_1 \otimes \mu_2) $ a simple tensor $f\otimes g$
is just the product $f\otimes g (x) = f(x_1) g(x_2)$ which is commonly identified with an
operator with the integral kernel  $f\otimes g$. Thus, in \eqref{eq:analysisoprepr} we may put
$$
T_{F}^{*} = T_{F_1}^{*} \otimes T_{F_2}^{*}.
$$

\par

Obviously, for a pair of continuous frames $F $ and $ G $ the condition (\ref{dual}) can be written as  $T_G T^*_F=I$ (in the weak sense).


\subsection{Non-simple Frames}
Let us digress a bit, and see if "everything is solved" now considering Theorem \ref{thm:main}.
In this subsection, we actually discuss the existence of non-simple tensor frames, therefore the result mentioned above does \emph{not} cover the full tensor frame theory (since it  concerns only simple tensors).
Let us stress that by Definition \ref{De:frame-Bessel}
it follows that not every frame in a tensor product Hilbert space
has to be represented as a (sequence of) simple tensor(s).
We will show that any continuous frame admits a non-simple dual frame.

Our first result shows that 
tensor Bessel sequences can be constructed with ranks different than $1$.

\begin{lem} Let $f_k(\omega)$ and $g_k(\nu)$ be continuous Bessel mappings in $\h_1$ and $\h_2$ respectively with bounds $B_k$ and $B_k'$ such
that $B := \sum_k B_k \cdot \sum_l B_l' < \infty$, then
$F (\omega,\nu) = \sum_k f_k (\omega) \otimes g_k (\nu)$ is a  Bessel mapping in $\h_1 \otimes \h_2$ with the Bessel bound $B$.
\end{lem}
\begin{proof} Note that
\begin{multline*}
|\langle \psi \otimes \phi,  F (x_1, x_2) \rangle |^2
 = | \langle \psi \otimes \phi, \sum_k f_k (x_1) \otimes g_k (x_2) \rangle |^2 \\
 =  | \sum_k  \langle \psi,  f_k (x_1)  \rangle  \langle  \phi, g_k (x_2) \rangle |^2.
\end{multline*}
Then we have
\begin{multline*}
\int  |\langle \psi \otimes \phi,  F (x_1, x_2) \rangle |^2 d\mu (x_1, x_2)
= \int  | \sum_k  \langle \psi , f_k (x_1)  \rangle  \langle  \phi, g_k (x_2) \rangle |^2 d\mu (x_1, x_2)  \\
\leq \sum_k \int  | \langle \psi , f_k (x_1)  \rangle|^2 d\mu (x_1) \sum_l  \int | \langle  \phi, g_l (x_2) \rangle |^2 d\mu ( x_2)  \\
\leq \sum_k B_{k} \| \psi\|^2 \cdot \sum_l B'_{l} \| \phi\|^2
\leq  (\sum_k B_{k} \sum_l B'_{l} ) \| \psi\|^2 \cdot \| \phi\|^2. \\
\end{multline*}
The result now follows by extension from simple tensors to all of $\h_1 \otimes \h_2$ (cf. \cite[Proposition 2.5]{RaNaDe}).
\end{proof}

A direct converse can never be true (consider e.g. any $f_k$ and $g_k = 0$).
%

So, we now know that non-simple Bessel sequence exist. If we now consider a fixed frame, do there dual frames exist, which are not necessarily of the form given by Theorem \ref{thm:main} (see also \cite{xxlfei1}).
More precisely, if $F_1 \otimes F_2$ is a frame for $\h$ with respect to $ (X, \mu)$, then we examine the existence of its dual frame $ G $
such that $ G \neq  G_1 \otimes G_2 $ for any $ G_1 \in \h_1,$ $ G_2 \in \h_2.$
Let us shortly digress from the logical order of results and rather use the proof of this result as a motivation for the next section.

We first recall that  a continuous frame $F$ is {\em redundant} if
$$
R (F) := \text{dim} (\range {T^{*} _{F}} ^\perp) >0.
$$
It has been observed that $R(F)$ depends on the underlying measure space $ (X, \mu)$. For example, if  $ (X, \mu)$ is non-atomic, then
$R(F) = \infty$. We refer to \cite{SpBa} for details.

\par

\begin{lem} \label{thm:nontensordual}
Let $ \text{dim} (\h_1 ), $ $ \text{dim} (\h_2 ) > 1 $,  and let $ F_1 \otimes F_2 $ be a redundant frame for $\h$. Then $ F_1 \otimes F_2 $ admits at least one non-simple tensor product dual.
\end{lem}

\begin{proof}

The idea is to follow the steps of the proof of \cite[Theorem 2.3]{WangLi}, and the case study examination given there.
This proof uses the fact that for $ T \in \h_1 \otimes \h_2 ,$ $\text{dim} \left( \range{T} \right) \leq 1 $ if and only if
$ T = f \otimes g$ for some  $ f \in \h_1 $ and $ g\in \h_2$ (\cite[Lemma 2.2]{WangLi}).
Replacing sums by integrations the proof of \cite[Theorem 2.3]{WangLi} can be generalized in a straightforward way, {\em if} we can show the tensor product version of \cite[Theorem 6.3.7]{ole}, i.e. a description of all dual tensor frames of a given tensor frame. This is Corollary \ref{thm:dualframes}.
\end{proof}

In order to make this proof complete we have to introduce the next section, which by itself answers an open question in frame theory.

\section{Full classification of dual continuous frames}

In this section we extend the well known classification of dual discrete frames \cite[Theorem 6.3.7]{ole}
to the continuous frames setting. This is a new result in continuous frame theory, and we apply it  to describe dual frames in the context of
tensor products.

It turned out that the theory of reproducing kernel Hilbert spaces (RKHS) provides convenient tools for the result in this section. The interplay
between RKHS and  frame theory is recently used in \cite{spexxl16} in the study of stable analysis/synthesis processes.
Recall, a Hilbert space
$\h$ is {\em a reproducing kernel Hilbert space} on the set $X$ if it is a subspace of the space of functions from $X$ to $\mathbb{C}$ such that
for every $x\in X$ the linear {\em evaluation functional} $ev_x :\h \rightarrow \mathbb{C}$ defined by
$ ev_x (f) = f(x) $ is bounded, cf. \cite{paulsen_raghupathi_2016}.
From the Hilbert space property we have that there exists
 an element $k_x \in \h$ such that $f(x) = \left< f, k_x \right>$, {\em the reproducing kernel}.

We will also use the following facts from the theory of RKHS:
any closed subspace of an RKHS is again a RKHS,  \cite[Theorem 2.5.]{paulsen_raghupathi_2016}.

We also need the following trivial result, which is a direct result of the definition:
\begin{cor} \label{cor:union1} Let $\h$ and $\g$ be subspaces of a normed space $X$, where the closures are RKHS. Then $\overline{\h \cup \g} = \overline{\h} \cup \overline{\g} $ is a RKHS.
\end{cor}

\begin{proof}
Let us denote by $k^{\h} _x$ and $k^{\g} _x$ the respective reproducing kernels for the closures.
Then, for $f \in  \overline{\h}$ we have that $\left| f (x) \right| \le \norm{k^{\h} _x} \norm{f}$, and $\left| f (x) \right| \le \norm{k^{\g} _x} \norm{f}$ for $f \in  \overline{\g}$, so that
$$ \left| f (x) \right| \le \max \left\{ \norm{k^{\h} _x}, \norm{k^{\g} _x} \right\} \norm{f}
\text{ for } f \in \overline{\h \cup \g}.$$
\end{proof}

The following result \cite[Proposition 11]{spexxl16}, which relates frames with RKHS, is needed later:

\begin{lem} \label{lem:rkhs}  If $F$ satisfies the lower frame inequality, then $ (\range {T^*_F}, \| \cdot \|) $ is a RKHS. Moreover, for any subspace $ \h_{K} $
of $ L^2 (X, \mu) $, the following are equivalent:
\begin{itemize}
\item[a)] $ \h_{K} $  is a RKHS.
\item[b)] There exists a continuous frame $ F $ such that $\range {T^*_F} = \h_{K} $.
\end{itemize}
\end{lem}

We are now ready to attack the main question in this section.
We first prove the continuous counterpart of \cite[Lemma 6.3.5]{ole}.

\begin{lem} \label{lem:dualcont1} Let $F(x)$ be a continuous frame for the Hilbert space $\h$.
Let $ e_k $ be an orthonormal basis for $\h$ and let $V:L^2(X,\mu) \rightarrow \h$ be a bounded left-inverse of $T_F ^{*} $,
 such that $\left(\ker {V}\right)^\perp$ is a reproducing kernel subspace of $L^2(X,\mu)$.
Then the dual frames of $F$ are precisely the functions $G(x)=\sum_{k \in K} \overline{V^* (e_k)(x)} e_k$.
\end{lem}

\begin{proof}
The function $G(x)$ is well defined if $\sum_{k \in K}|{V^*(e_k)(x)}|^2 < \infty$.
By \cite[Proposition 6]{spexxl16}, $\sum_{k \in K}|{V^*(e_k)(x)}|^2 < \infty$
if $V^*(e_k)$ is a discrete Bessel sequence in the RKHS $\left( \ker {V} \right)^\perp$.
Now, from the proof of \cite[Proposition 5.3.1]{ole} it follows that this is true.

Next, following \cite[Proposition 21]{spexxl16}, we have
\begin{multline*}
\langle f,g \rangle_\h = \langle V T_F ^{*} f,g \rangle_\h   =  \langle T_F ^{*} f, V^*g \rangle_{L^2(X, \mu)} \\
 =   \langle  T_F ^{*} f, \sum_{k\in K} \langle g,e_k \rangle V^*e_k \rangle_{L^2(X,\mu)}
 =  \langle  T_F ^{*} f,  \langle g,\sum_{k \in K} \overline{V^*(e_k)(.)} \cdot e_k \rangle \rangle_{L^2(X, \mu)} \\
 =  \langle T_F ^{*} f, T_G ^{*} g \rangle_{L^2(X, \mu)}.
\end{multline*}

On the other hand, let $ G_0(x)$ be a frame and set $V=T_{G_0}$. We have that $  \ker {V} ^\perp = \range {T_{G_0} ^{*} }$,
which is a reproducing kernel Hilbert space by Lemma \ref{lem:rkhs}. Thus
\begin{multline*}
T_G ^{*}  (g)  (x) = \langle g, G(x) \rangle_\h  =   \langle g, \sum_{k \in K} V^* e_k (x) e_k \rangle_\h \\
  =  \sum_{k \in K}V^* e_k(x)  \langle g, e_k\rangle_\h =  \left( V^*g \right) (x).
\end{multline*}
Thus $ T_G ^{*} = T_{G_0}  ^{*} $ a.e. and so $ G(x) = G_0 (x)$ a.e.
\end{proof}

Note that the left-inverse $V$ in Lemma \ref{lem:dualcont1} can never be invertible.
(Because then $\ker {V} ^\perp = L^2(\mathbb{R})$, which is not a RKHS).

The continuous version of \cite[Lemma 6.3.6]{ole} can now be given as follows.

\begin{lem} \label{lem:637cont} Let $F(x)$ be a continuous frame for $\h$. 
The bounded left-inverses of $T_F ^*$, i.e. $V T_F^* = \identity{\h}$, with $\kernel{V}^\bot$ being a RKHS  are precisely the operators of the form
\begin{equation} \label{leftinverseV}
V = S_F ^{-1} T_F + W \left( \identity{\h} - T_F ^* S_F ^{-1} T_F \right),
\end{equation}
where $W: L^2(X,\mu) \rightarrow \h$ is a bounded operator with $\kernel{W}^\bot$ being a RKHS.
\end{lem}

\begin{proof}
The proof of  \cite[Lemma 6.3.6]{ole} can be used directly in the sense that all left-inverses $V$ can be exactly represented by \eqref{leftinverseV}.
It remains to prove the transfer of the RKHS property.
\par

Consider the mapping $W_0 : \kernel{T_F} \rightarrow \h$, defined by $W_0 := W \pi_{\kernel{T_F}}$. The operator $W$ can then be written as $W = W_0 \pi_{\kernel{T_F}} + W \pi_{\range{T_F^*}}$.

By the assumptions we have that
$$ V = S_F^{-1} T_F + W \pi_{\kernel{T_F}} = S_F^{-1} T_F + W_0 \pi_{\kernel{T_F}}, $$
and
$$ V^* = T_F^* S_F^{-1} + \pi_{\kernel{T_F}} W^* = T_F^* S_F^{-1} + \pi_{\kernel{T_F}} W_0^*, $$

Clearly, $\kernel{W_0} = \kernel{W} \cap {\kernel{T_F}}$. In particular, $\kernel{W}^\bot \subseteq \kernel{W_0}^\bot$ and $\kernel{W_0}^\bot = \kernel{W}^\bot \cup \kernel{T_K}^\bot = \kernel{W}^\bot \cup \range{T_K^*}$. This tells us that (by Corollary \ref{cor:union1}) that $\kernel{W}^\bot$ is a RKHS if and only if $\kernel{W_0}^\bot$ is.

Additionally, by construction, $\kernel{W_0} \subseteq \kernel{V}$. Therefore, if we assume that
$\kernel{W}^\bot$ is a RKHS, then $\kernel{W_0}^\bot$ is a RKHS, and so is $\kernel{V}^\bot$. This shows the first direction.

For the opposite direction, assume that $\kernel{V}^\bot$ is a RKHS. Let $k^V_x$ be the kernel of $\range{V^*}$ and $k^F_x$ the one on $\range{T_F^*}$. Then
\begin{eqnarray*}
\left| \left( W_0^* f \right) (x) \right|  & = &  \left| \left( V^* f \right) (x) - \left( T_F^* S_F^{-1} f \right) (x) \right| \\
& \le & \left| \left( V^* f \right) (x) \right| + \left| \left( T_F^* S_F^{-1} f \right) (x) \right| \\
& \le & \norm{k^V_x} \norm{V^* f} + \norm{k^F_x} \norm{T_F^* S_F^{-1} f}  \\
& \le & \left(\norm{k^V_x} \norm{V} + \norm{k^F_x} \frac{1}{A} \right) \norm{f},  \\
\end{eqnarray*}
where $A$ is the lower frame bound of $F(x)$.
This shows the other direction.

%
%

\end{proof}

Next we prove the continuous frame counterpart of  \cite[Theorem 6.3.7]{ole}.

\begin{thm} \label{olecontinuous}
Let $F$ be a continuous frame for $\h$. The dual frames of $F$ are precisely the functions
\begin{equation} \label{eq:classdual1}
G (x) = S^{-1}_{F} F(x) + \Theta (x)- \int \left< S^{-1} F(x) , F(y) \right> \Theta(y) d \mu (y) ,
\end{equation}
where $\Theta$ is a Bessel mapping.
\end{thm}
\begin{proof}
Equation \eqref{eq:classdual1} is equivalent to
$$ G (x) = {
\left( T_{\widetilde F} - T_\Theta \right) T_F^* S^{-1} F \left(x \right) + \Theta(x).
}
$$
By the construction it is a Bessel mapping as a sum of Bessel mappings. Note that a bounded operator applied to a Bessel mapping again gives a Bessel mapping.
Since $T_G T_F^* = \identity{\h}$, it is a dual frame.

On the other hand let $G_0$ be dual frame of $F$. Then $V = T_{G_0}$ is a left inverse of $T_F^*$, where $\kernel{V}^\bot$ is a RKHS.
By Lemma \ref{lem:637cont} it follows that
$$ V = S_F^{-1}T_F + W(I - T_F^* S_F^{-1} T_F),
$$
where $W$ is a bounded operator with $\kernel{W}^\bot$ being a RKHS. By Lemma \ref{lem:dualcont1} we have
$G(x)=\sum_{k \in K} \overline{V^* (e_k)(x)} e_k$.
Therefore
\begin{eqnarray*}
G(x) & =  \sum_{k \in K} \overline{T_F^* S_F^{-1} (e_k)(x)} e_k  & + \sum_{k \in K} \overline{W^* (e_k)(x)} e_k   \\
& & -  \sum_{k \in K} \overline{T_F^* S_F^{-1} T_F W^* (e_k)(x)} e_k \\
& =   \sum_{k \in K} \underbrace{\overline{T_F^* S_F^{-1} (e_k)(x)}}_{= \widetilde F (x)} e_k  & + \left(\identity{\h} - T_F^* S_F^{-1} T_F \right)
\sum_{k \in K} \underbrace{\overline{W^* (e_k)(x)} e_k}_{=: \Theta(x)}.
\end{eqnarray*}
The sequence  $W^* (e_k)$ is a Bessel sequence in the RKHS $\range{W^*} = \kernel{W}^\bot$ and so $\Theta(x)$ is well-defined. Furthermore
\begin{eqnarray*}
\left< f, \Theta(x) \right> = \left< f,   \sum_{k \in K} \overline{W^* (e_k)(x)} e_k \right> = \sum_{k \in K} W^* (e_k)(x) \left< f,    e_k \right> \\
= ev_x \left( \sum_{k \in K} W^* (e_k)  \left< f,    e_k \right> \right) =
ev_x \left( W^* \sum_{k \in K} \left< f,    e_k \right> e_k \right) = ev_x \left(W^* f \right).
\end{eqnarray*}
Therefore
$$
\int  |\langle f, \Theta(x) \rangle |^2 d\mu (x) = \norm{W^*f}_{L^2(X,\mu)} \le \norm{W}_{Op} \norm{f}_{L^2(X,\mu)}, $$
and $\Theta(x)$ is a continuous Bessel mapping.
%
%
%
%
\end{proof}

Like in the discrete setting \cite{zak} this can be reformulated as
\begin{cor}
Let $F$ be a continuous frame for $\h$. The dual frames of $F$ are precisely the functions
\begin{equation*} \label{eq:classdual1cor}
G (x) = S^{-1}_{F} F(x) + \Theta (x),
\end{equation*}
where $\Theta$ is a Bessel mapping with $\range{T_\Theta}^* \subseteq \kernel{T_F}$.

\end{cor}

Adapting Theorem \ref{olecontinuous} to the tensor frame setting we reach the following:
\begin{cor}  \label{thm:dualframes}
Let $ F_1 \otimes F_2 $ be a frame for $\h$. Then the dual frames of $ F_1 \otimes F_2 $ are precisely the families of the form
\begin{align*}
S^{-1} _{F_1} F_1 (x_1) & \otimes S^{-1} _{F_2} F_2 (x_2)  + W (x_1, x_2)  \\
& - \int_X \langle S^{-1} _{F_1} F_1 (x_1), F (y_1)\rangle
\langle S^{-1} _{F_2} F_2 (x_2), F (y_2)\rangle  W(y_1, y_2) d\mu (y),
\end{align*}
where $ W $ is a Bessel mapping in $\h$.
\end{cor}

Now let us come back to Lemma \ref{thm:nontensordual}: To show that there are non-simple dual tensor frames, one has to find a non-simple Bessel mapping $W$ such that the result is also non-simple. This can be done as in the case study of the proof of \cite[Lemma 2.2]{WangLi}.

\section{Tensor Product Continuous Frame Multipliers} \label{sec:multipliers}

Gabor multipliers \cite{feic} led to the introduction of
Bessel and frame multipliers for abstract Hilbert spaces. These operators are
defined by a fixed multiplication pattern (the symbol) which is
inserted between the analysis and synthesis operators
\cite{xxlmult1,peter2,peter3}.
This section is inspired by the continuous frame multipliers studied in \cite{BBR}.
We are interested in the tensor product setting as follows.

\begin{defn}\label{definitioncontframemult}
Let $\h$ be the tensor product $\h = \h_1 \otimes \h_2 $ of complex Hilbert spaces, and $(X, \mu) = ( X_1 \times X_2, \mu_1 \otimes \mu_2)$ be the
product of measure spaces with $\sigma-$finite positive measures $\mu_1, \mu_2.$
Also, let  $F $ and $ G $  be Bessel mappings for $\h$ with respect to $(X,\mu)$ and $m:X\rightarrow \mathbb{C}$ be a
measurable function. The operator
$\textbf{M}_{m,F,G}:\h\rightarrow\h$ weakly defined by
\begin{equation}
\label{de:operatorM}
\langle \textbf{M}_{m,F,G} \vec{f}, \vec{g}\rangle =
 \int_{X} m(x) \langle \vec{f}, F (x) \rangle
\langle G (x),\vec{g} \rangle d\mu(x),
\end{equation}
for all $\vec{f}, \vec{g} \in\h$, is called  {\em tensor product continuous Bessel multiplier} of $F$
and $G$ with respect to the  {\em symbol} $m$. If, in addition, $F$ and $G$ are continuous frames, then $\textbf{M}_{m,F,G}$
given by \eqref{de:operatorM} is called {\em tensor product continuous frame multiplier}.
\end{defn}

Eq.\eqref{de:operatorM} is equivalent to the weak formulation of
\begin{equation}
\label{de:operatorMweak}
\textbf{M}_{m,F,G} \vec{f}:=\int_{X}m(x)\langle \vec{f}, F(x)\rangle G(x)d\mu(x).
\end{equation}


\begin{rem}
If $ m \equiv 1$ and  $F $ and $ G $ are Bessel mappings for $\h$ with respect to $(X,\mu)$, then $\textbf{M}_{1,F,G}$
given by \eqref{de:operatorM} is a well-defined and bounded sesquilinear form on $\h$, which could be called the {\em cross-frame operator}.

If, in addition, the corresponding operator given  by
\eqref{de:operatorMweak} has a bounded inverse, then $(F,G)$ is a reproducing pair for $\h$ in the sense of \cite{spexxl16}
(when the definition of reproducing pairs is suitably interpreted for tensor product of Hilbert spaces).

If $(F,G)$ is a dual pair of continuous frames (cf. Definition \ref{def:dualframes}), then  $\textbf{M}_{1,F,G}$ given by
\eqref{de:operatorMweak} is the identity operator (and vice-versa, as we have assumed the Bessel property).
\end{rem}

If $\textbf{M}_{m,F,G} $ is given by \eqref{de:operatorM}, then it immediately follows that $(\mathbf{M}_{m,F,G})^*=\mathbf{M}_{\overline{m},G,F}$, cf. \cite[Proposition 3.4]{BBR}.

\par

\begin{lem}\label{tar}
Let $F$ and $G$ be as in Definition \ref{definitioncontframemult}, with the Bessel bounds $B_F$ and $B_G$ respectively. If $m\in
L^{\infty}(X,\mu)$, then the continuous tensor product Bessel multiplier
$\textbf{M}_{m,F,G}$  given by \eqref{de:operatorM}
is well defined and bounded with
\[
\|\textbf{M}_{m,F,G}\|\leq \|m\|_\infty \sqrt{B_F B_G}.
\]
\end{lem}

\begin{proof} The proof is a modification of the proof of \cite[Lemma 3.3]{BBR} for the  case of tensor products, and is therefore omitted.
\end{proof}

Here and in what follows the norm in Lebesgue spaces $ L^{p}(X,\mu)$, $1\leq p\leq \infty $ is denoted by
$\| \cdot \|_p $. As usual, we shorten notation by setting $ \| \cdot \| = \| \cdot \|_2.$

If $m(x)>0$ a.e., then for any Bessel
mapping $F$ the multiplier $\textbf{M}_{m,F,F}$ is a positive
operator, and if $m(x)\geq \delta > 0$ almost everywhere for some positive constant $\delta$ and
$\|m\|_{\infty}<\infty$ then $\textbf{M}_{m,F,F}$ is just the
frame operator of $\sqrt{m}F$ and so it is positive, self-adjoint
and invertible, cf. \cite{BBR}.

\par

By using  analysis and synthesis operators for $F$ and $G$, it is  easy to  see that
\begin{equation}
\label{rep1}\textbf{M}_{m,F,G}=T_G \circ D_m \circ T^*_F
\end{equation}
where $D_m:L^{2}(X,\mu)\rightarrow L^{2}(X,\mu) $ is given by
$(D_m \varphi)(x)=m(x)\varphi(x)$. If $m\in L^{\infty}(X,\mu)$, then $D_m$ is bounded and
$\|D_m\|=\|m\|_{\infty}$, \cite{conw1}.

\par

If  $m\in L^{\infty}(\RR^d, dx) $, then  \cite[Proposition 3.6]{BBR} implies that
the multiplication operator $D_m$ on $ L^2(\RR^d, dx)$ (with $dx$ denoting the Lebesgue measure)
is compact if and only if $m \equiv 0$. This
constitutes an important difference between the discrete and the continuous case, see \cite{peter2}.
To prove sufficient conditions for compactness of tensor product
continuous frame multipliers a different approach than in the discrete setting has to be taken.
We closely follow the approach suggested in \cite{BBR}.

\vspace{5mm}

\subsection{Compact Multipliers}

Recall, a mapping $F$ is called norm bounded on $( X, \mu)$ if  there exists a constant $C > 0$ such that $\| F(x) \| \le C$ for almost every $x \in X$.
Furthermore, the  support of  measurable function $m:X\rightarrow \mathbb{C}$ is of a finite measure if
there exists a subset $K \subseteq X$ with $\mu(K) < \infty$ such that $m(x) = 0$ for almost every $x \in X \setminus K$.

We can formulate \cite[Theorem 3.7]{BBR} in the tensor product setting:
\begin{thm} \label{Th:compact-01}
Let $F$ and $G$ be as in Definition \ref{definitioncontframemult}, and let either $F$ or $G$ be norm bounded. If $m: X\rightarrow \mathbb{C}$ is a
(essentially) bounded measurable function with support of finite measure, then $\textbf{M}_{m,F,G}$ given by \eqref{de:operatorM} is a compact
operator.
\end{thm}

The conclusion of  Theorem \ref{Th:compact-01} remains the same if, instead of having the support of finite measure,
we assume that $m: X \to \mathbb{C}$ vanishes at infinity, i.e. for every $\varepsilon > 0$ there is a set of finite measure $K = K(\varepsilon)
\subseteq X$, $\mu(K) < \infty$, such that $m(x)\le \varepsilon$ for almost every $x \in X \setminus K$ (cf.  \cite[Corollary 3.8]{BBR}).

If, in addition, we assume that \emph{both} $F$ and $G$ are norm bounded, then we have the following trace class, and Schatten $p-$class result which is a reformulation of \cite[Theorems 3.10 and 3.11]{BBR} to our setting:
\begin{thm}  \label{sec:schatten1}
Let $F$ and $G$ be as in Definition \ref{definitioncontframemult} which are norm
bounded  with norm bounds $L_F$ and $L_G$, respectively. Then the following is true:

\begin{enumerate}
\item If $m \in L^1(X,\mu)$, then
$\textbf{M}_{m,F,G}$ is  a trace class operator with the trace norm estimate given by
$$\| \textbf{M}_{m,F,G} \|_{\mathcal{S}_1} \le \|m\|_{1}L_F L_G.$$
\item  If $m \in L^p(X,\mu)$, $1 < p < \infty$,
then $\textbf{M}_{m,F,G}$ belongs to the Schatten p-class $\mathcal{S}_p(\h)$, with norm estimate
\[
\| \textbf{M}_{m,F,G} \|_{\mathcal{S}_p} \leq \| m \|_p \left( L_F L_G \right)^{\frac{1}{p}}\left(B_F B_G \right)^{\frac{p-1}{2p}}.
\]
\end{enumerate}
\end{thm}

We omit the proof since it follows by  slight modifications of the proofs of  \cite[Theorems 3.10 and 3.11]{BBR}.

Recall, if $A\in \mathcal{S}_1(\h)$, then its trace is defined to be
$$
\text{Tr}_{\h} (A) = \sum_n \langle A e_n, e_n \rangle,
$$
for any ONB in $\h$. We have $ |\text{Tr}_{\h} (A)| \leq \| A \|_{\mathcal{S}_1} $,  with the equality if $A$ is a positive operator.

For tensor product Hilbert spaces $ \h = \h_1 \otimes \h_2 $, the following  partial trace theorem holds.

\begin{thm} \label{th:partialtrace}
Let $\h$ be a tensor product Hilbert space $ \h = \h_1 \otimes \h_2 $, and let $A\in \mathcal{S}_1(\h)$.
Then there is a continuous and linear map
\begin{equation} \label{eq:parttracemap}
T: \mathcal{S}_1(\h) \rightarrow \mathcal{S}_1(\h_1)
\end{equation}
such that the following properties hold:
\begin{equation} \label{eq:parttraceformula}
T(A_1 \otimes A_2) = A_1 \text{{\em Tr}}_{\h_2} (A_2), \qquad \forall A_j \in \mathcal{S}_1(\h_j), \;\;\; j=1,2,
\end{equation}
\begin{equation} \label{eq:parttraceproprerty}
\text{{\em Tr}}_{\h_1} (T(A)) = \text{{\em Tr}}_{\h} (A), \qquad \forall A \in \mathcal{S}_1(\h).
\end{equation}

\end{thm}

Proof of Theorem \ref{th:partialtrace} is contained in the proof of  \cite[Theorem 26.7]{BB}, and therefore omitted.

\par

If $T$ is the  mapping given by \eqref{eq:parttracemap}, then
$T(A)$ is called the {\em partial trace} of $A$ with respect to $\h_1$. In a similar way we may define the partial trace of $A$ with respect to
$\h_2 $.

\par

In Section \ref{sec:quantum} we will use the following simple consequence of Definition \ref{definitioncontframemult} and   Theorem \ref{th:partialtrace}.

\begin{cor}\label{cor:partialtrace}
Let  $m_j $ be measurable functions on $X_j$, let $F_j$ and $G_j$ be continuous Bessel mappings (frames) for $\h_j$, $ j=1,2$, and
let $m= m_1 \otimes m_2 $, $ F = F_1 \otimes F_2 $, and  $ G = G_1 \otimes G_2 $. If
$ \textbf{M}_{m,F,G}  \in  \mathcal{S}_1(\h_1 \otimes \h_2)$,
then its partial trace $T (\textbf{M}_{m,F,G} )$  with respect to $\h_1$ is a continuous Bessel (frame) multiplier given by
$$
T (\textbf{M}_{m_1,F_1,G_1}  \otimes \textbf{M}_{m_2,F_2,G_2} ) = \textbf{M}_{m_1,F_1,G_1}  \text{{\em Tr}}_{\h_2} (\textbf{M}_{m_2,F_2,G_2} ),
$$
i.e. it is a trace class operator of "the same form" as
$ \textbf{M}_{m,F,G} $.


Similar holds for the  partial trace of  $\textbf{M}_{m,F,G}$  with respect to $\h_2$.
\end{cor}

\section{Bilinear localization operators} \label{sec:STFTwavmult0}

In this section, we reveal bilinear localization operators as examples of tensor product continuous frame multipliers.
In the case of short-time Fourier transform multipliers (STFT multipliers),
the results from Section \ref{sec:multipliers} are in line with those of \cite{CKasso, Teof2018}, while
their interpretation in the case of wavelet multipliers (Calder\'on--Toeplitz operators)
and mixed STFT/wavelet multipliers seems to be new, although their "linear" counterparts are
well studied, see e.g. \cite[Section 3.4]{BBR} for a brief survey.
In addition, let us mention that the continuity properties of multipliers for the  ridgelet transform given in \cite{LiWong}
can be derived from the results of \cite{BBR}.

\par

STFT multipliers, also known as time-frequency localization operators,
are used in signal analysis as a mathematical tool to extract specific features
of a signal from  its  phase space representations, \cite{da88}. In other contexts, they have been used as a quantization procedure \cite{Berezin71}, or as an approximation of pseudodifferential operators, cf. \cite{CRodino} and the references given there.

\par

We first recall some necessary facts.

\par

Let $T_xf(\cdot):=f(\cdot -x)$, $M_\omega  f(\cdot ):=e^{2\pi i\omega \cdot }f(\cdot)$, and  $D_a f(\cdot) :=|a|^{-d/2} f(\frac{\cdot}{a})$, denote
translation, modulation, and  dilation operators,  respectively, $x,\omega \in \mathbb{R}^{d}$,
$a \in \mathbb{R}\setminus \{ 0\}. $ These operators are unitary on $ L^2 (\mathbb{R}^d)$, and we use the notation
\begin{align*}
\pi(x,\omega) & = M_\omega T_x, \qquad & \text{for} \qquad & (x,\omega)\in \mathbb{R}^{2d}, \\
 \paff (b,a) & = T_b D_a, \qquad  &\text{for} \qquad & (b,a) \in \mathbb{R}^{d} \times (\mathbb{R}\setminus \{ 0 \}).
\end{align*}

\par

Let $ \hat g$ denote the Fourier transform of $g\in  L^1 (\mathbb{R}^d)$ given by
${\hat   {g}}(\omega)$ $ = \int g(t)e^{-2\pi i t\omega}dt$. This definition extends to  $ L^2 (\mathbb{R}^d)$ by density arguments. We say that
$g \in L^{2}(\mathbb{R}^d)$ is an {\em admissible wavelet} if
\begin{equation} \label{eq:admissiblewavelet}
0 < C_{g}:=\int_{\mathbb{R}^{d}} \frac{|\hat{g}(\omega)|^{2}} {|\omega|}\,d\omega< +\infty.
\end{equation}

\begin{defn}\label{Def:transforms}
Let $ g\in L^{2}(\mathbb{R}^d)\setminus\{0\}$. The
short-time Fourier transform (STFT) of a function $f\in L^{2}(\mathbb{R}^d)$ with respect to the window function $g $ is given
by
$$
V_g f(x,\omega) :=    \int_{\mathbb{R}^d}  f(t)\, {\overline {g(t-x)}} \, e^{-2\pi i\omega t}\,dt =
\langle f,M_\omega T_x g \rangle = \langle f, \pi(x,\omega) g \rangle,
$$
$ (x,\omega)\in \mathbb{R}^{2d}$.
If, in addition, \eqref{eq:admissiblewavelet} holds, i.e. $g$ is an admissible wavelet,
then the (continuous) wavelet transform  of  $f\in L^{2}(\mathbb{R}^d)$ with respect  $g $ is given
by
$$
W_{g}(f)(b,a):= \int_{\mathbb{R}^{d}} f(t)\frac{1}{|a|^{\frac{d}{2}}}\overline{g(a^{-1} (t-b )}\,dt =
\langle f,  T_b D_a g \rangle = \langle f, \paff (b,a) g\rangle,
$$
$  b\in \mathbb{R}^{d}, a \in \mathbb{R}\setminus \{ 0 \}.$
\end{defn}

\par

Definition \ref{Def:transforms} can be extended to various spaces of (generalized) functions, but we focus our attention here to
$ L^{2}(\mathbb{R}^d) $ to make the exposition of our main ideas more transparent.

\par

By the orthogonality relation (see e.g. \cite[Theorem 3.2.1]{Grobook})
\begin{equation} \label{eq:STFTortrel}
 \langle V_{g_{1}} f_{1}, V_{g_{2}} f_{2} \rangle
=\langle f_{1}, f_{2}\rangle \overline{\langle g_{1},g_{2}\rangle}, \;\;
f_{1}, f_{2},\in L^{2}(\mathbb{R}^d), \;\;
 g_{1}, g_{2} \in L^{2}(\mathbb{R}^d) \setminus \{ 0 \},
\end{equation}
if $  g_{1} = g_{2} = g$ it follows that $ \pi(x,\omega) g $ is  a continuous tight frame for $L^{2}(\mathbb{R}^d)$
with respect to $ (\mathbb{R}^d \times \mathbb{R}^d, dx d\omega)$ and with bound $\|g\|^2$, for any $ g\in L^{2}(\mathbb{R}^d)\setminus\{0\}$.
If  $ \| g \| = 1$, then we have a continuous Parseval frame.

Likewise, for the wavelet transform the following  orthogonality relation holds:
\begin{equation} \label{eq:WTortrel}
\int_0 ^\infty \int_{\mathbb{R}^{d}} W_{g_1}(f_1)(b,a) \overline{ W_{g_2}(f_2)(b,a)} \frac{db da}{a^{d+1}} = C_{g_1,g_2} \langle f_{1},
f_{2}\rangle,
 \;\;
f_{1}, f_{2}\in L^{2}(\mathbb{R}^d),
\end{equation}
if  $g_{1}, g_{2}\in L^{2}(\mathbb{R}^d)$ are such that for almost all $\omega \in \mathbb{R}^d$  with $|\omega| = 1$,
$$
\int_0 ^\infty | \hat g_1 (s \omega) | | \hat g_2 (s \omega) | \frac{ds}{s} < \infty,
$$
and the constant $ C_{g_1,g_2} $ given by
$$
 C_{g_1,g_2} := \int_0 ^\infty \overline{ \hat g_1 (s \omega) } \hat g_2 (s \omega ) \frac{ds}{s}
$$
is finite, non-zero, and independent on $\omega$, cf. \cite[Theorem 10.2]{Grobook}.

If $g\in L^{2}(\mathbb{R}^d)$ is an admissible and rotation invariant function, then
the orthogonality relation holds for $g=g_1=g_2$, and
$ \paff(b,a) g $ is a continuous tight frame for $L^{2}(\mathbb{R}^d)$ with respect to
$  ( \mathbb{R}^{d} \times \mathbb{R}\setminus \{ 0 \}, \frac{db da}{a^{d+1}}) $.
The frame bound is $1/C_{g,g}$, and if $g$ is suitably normed so that $C_{g,g} = 1$, then we have a continuous Parseval frame.

\par

Related continuous frame multipliers, called STFT and Calderon-Toeplitz multipliers, were discussed in \cite{BBR}. Here
we consider the tensor product space $ \h = L^{2}(\mathbb{R}^d) \otimes L^{2}(\mathbb{R}^d)$ instead.

\par

If $ \vec f, \vec \varphi \in \h$, then
$$
 V_{\vec \varphi } \vec f(x,\omega)=  \langle \vec f, \pi(x,\omega)  \vec \varphi \rangle =
   \int_{\mathbb{R} ^{2d}}  \vec f(t)\, \overline{ \pi(x,\omega)  \vec \varphi (t)}\,dt, \qquad  x,\omega \in \mathbb{R}^{2d},
$$
and if $ \vec \varphi (t) = \varphi_1 \otimes \varphi_2 (t)  =  \varphi_1 (t_1) \varphi_2 (t_2), $ $ t = (t_1, t_2)  \in \mathbb{R}^{d} \times
\mathbb{R}^{d},$ then  $  V_{\vec \varphi } $ acts on a simple tensor $f_1 \otimes f_2 \in \h$ as
\begin{multline} \label{STFTtensor}
V_{\varphi _1 \otimes \varphi _2} (f_1 \otimes f_2) (x,\omega)
= \int_{\mathbb{R}^{2d}} (f_1 \otimes f_2) (t) \overline{  \pi(x,\omega) \varphi _1 \otimes \varphi _2 (t)} dt \\
= \iint_{\mathbb{R}^{d}} (f_1 \otimes f_2) (t) \overline{ \pi(x_1 ,\omega_1 ) \varphi _1 (t_1)  \pi(x_2 ,\omega_2 ) \varphi _2 (t_2)} dt_1 dt_2.
\end{multline}

\begin{lem} \label{lm:STFTtensprod}
Let $ \h = L^{2}(\mathbb{R}^d) \otimes L^{2}(\mathbb{R}^d)$, and
$ \vec \varphi  = \varphi_1 \otimes \varphi_2 \in  \h \setminus \{ 0 \}$. Then
$$
\pi(x,\omega)  \vec \varphi (t) = \pi(x_1 ,\omega_1 ) \varphi _1 (t_1)  \pi(x_2 ,\omega_2 ) \varphi _2 (t_2)
$$
is a continuous  tight frame for the tensor product space $\h$ with respect to
$ (\mathbb{R}^d \times \mathbb{R}^d, dx d\omega)$.
\end{lem}

\begin{proof}
From \eqref{eq:STFTortrel} and \eqref{STFTtensor} it follows that  the orthogonality relation holds for simple tensors.
Now by \cite[Proposition 2.5]{RaNaDe}  the orthogonality relation can be extended to  $\h$, and we conclude that
$$
\pi(x,\omega)  \vec \varphi (t) =  \pi(x_1 ,\omega_1 ) \varphi _1 (t_1)  \pi(x_2 ,\omega_2 ) \varphi _2 (t_2), \quad
x, \omega \in  \mathbb{R}^{2d},
$$
with $t = (t_1, t_2)  \in \mathbb{R}^{d} \times \mathbb{R}^{d},$ is a continuous  tight frame for
$\h$, i.e.
$$
\langle \vec f, \pi(x,\omega) \vec \varphi \rangle = \|\vec f \| \| \vec \varphi \|.
$$
If, in addition, $\vec \varphi \in \h$ is chosen so that $ \| \vec \varphi \| = 1$, then $\pi(x,\omega)  \vec \varphi$ is a Parseval frame.
\end{proof}


\par

Let $\vec \varphi =\varphi_1 \otimes \varphi_2, \vec \phi = \phi_1 \otimes \phi_2 \in \h,$
$ \| \vec \varphi \| =  \| \vec \phi \| = 1$, and let $m : \mathbb{R}^{4d} \mapsto
\mathbb{C}$ be a measurable function.
Then the tensor product continuous frame multipliers of the form
$
M_{m, \pi(x,\omega)  \vec \varphi, \pi(x,\omega)  \vec \phi}
$
can be identified with  bilinear localization operators considered in \cite{CorGro2003, Teof2018} (see Remark 1.2 in \cite{Teof2018}), i.e.
\begin{equation} \label{eq:Mblop}
\langle M_{m, \pi(x,\omega)  \vec \varphi, \pi(x,\omega)  \vec \phi} \vec{f}, \vec{g} \rangle  =
\langle  m V_{\varphi _1 \otimes \varphi _2} (f_1 \otimes f_2) ,
V_{\phi_1 \otimes \phi_2} (g_1 \otimes g_2)   \rangle,
\end{equation}
$ f_1,f_2,g_1,g_2 \in   L^{2}(\mathbb{R}^d) $.
The function $m$ is commonly called {\em the symbol} of the operator
$M_{m, \pi(x,\omega)  \vec \varphi, \pi(x,\omega)  \vec \phi}$.

Certain Schatten class properties of bilinear localization operators given by
\eqref{eq:Mblop} can be deduced from their linear counterparts given in e.g.
\cite{CorGro2003,CPRT2}. In these investigations, localization operators are interpreted as Weyl pseudodifferential operators.
We note that these results extend results from Section \ref{sec:multipliers} in the considered special case.
However, we present here a simple alternative proof of related particular result for the linear case given in \cite{CorGro2003}.

\begin{prop} \label{prop:STFTSp}
Let $ \h = L^{2}(\mathbb{R}^d) \otimes L^{2}(\mathbb{R}^d)$, and let
$ \varphi_1, \varphi_2, \phi_1, \phi_2  \in  L^{2}(\mathbb{R}^d) \setminus \{ 0 \}$.
If $ m  \in L^p (\mathbb{R}^{2d}),$ $ 1\leq p<\infty,$ then
$M_{m, \pi(x,\omega)  \vec \varphi, \pi(x,\omega)  \vec \phi} $ given by \eqref{eq:Mblop} belongs to
Schatten class $ \mathcal{S}_p(\h)$.
\end{prop}

\begin{proof}
By Lemma \ref{lm:STFTtensprod} it follows that $F = \pi(x,\omega)  \vec \varphi$ and $G = \pi(x,\omega)  \vec \phi $ are
continuous  tight frames for $\h.$ Thus $M_{m, F, G} $ is a tensor product continuous frame multiplier, and by Theorem \ref{sec:schatten1}
it follows that $M_{m, F, G} \in \mathcal{S}_p(\h)$.
\end{proof}

Next we discuss bilinear Calder\'on-Toeplitz operators. To that end we consider time-scale shifts, and the left Haar measure $\mu = dbda/a^{d+1}$.

Let  $ \varphi_1,  \varphi_2 $ be admissible rotation invariant  wavelets, $\vec\varphi =  \varphi_1 \otimes \varphi_2 \in \h,$ and let $ \vec f \in \h$. Then the tensor product continuous wavelet transform is given by
$$
W_{\vec\varphi}( \vec f)(b,a)  = \langle \vec f, \paff (b,a) \vec\varphi \rangle, \qquad
  b\in \mathbb{R}^{2d}, a \in \mathbb{R}^2 \setminus \{ 0 \}.
$$
It acts on a simple tensor $f_1 \otimes f_2 \in \h$ as
\begin{multline} \label{WTtensor}
W_{\vec\varphi} (f_1 \otimes f_2) (b,a)  =
W_{\varphi_1} (f_1 ) \otimes W_{\varphi_2} (f_2) (b_1, b_2, a_1, a_2) \\
= \langle f_1, \paff (b_1,a_1) \varphi_1 \rangle   \langle f_2, \paff (b_2,a_2) \varphi_2 \rangle
\end{multline}
where $b = (b_1, b_2) \in \mathbb{R}^{2d} $,  $ a= (a_1, a_2 ) \in  \mathbb{R}^2 \setminus \{ 0 \}$,
and
$$
\paff (b,a) \vec\varphi = \paff (b_1,a_1) \varphi_1 \otimes \paff (b_2,a_2) \varphi_2.
$$

\begin{lem} \label{lm:WTtensprod}
Let $ \h = L^{2}(\mathbb{R}^d) \otimes L^{2}(\mathbb{R}^d)$, and
$\vec\varphi =  \varphi_1 \otimes \varphi_2 \in \h,$ where $ \varphi_1$ and $\varphi_2 $ are admissible rotation invariant  wavelets. Then
$$
\paff (b,a)  \vec \varphi (t), \quad
b\in \mathbb{R}^{2d}, a \in \mathbb{R}^2 \setminus \{ 0 \},
$$
is a continuous  tight frame for the tensor product space $\h$ with respect to
$ (\mathbb{R}^{2d} \times\mathbb{R}^2 \setminus \{ 0 \}, \frac{db da}{a^{d+1}})$.
\end{lem}

The proof is similar to the proof of Lemma \ref{lm:STFTtensprod}, and therefore omitted.

If $m : \mathbb{R}^{2d} \times  \mathbb{R}^2 \setminus \{ 0 \} \mapsto \mathbb{C}$ is a measurable function, then the tensor product continuous frame multipliers of the form
\begin{equation} \label{eq:Mblopaffine}
M_{m,  \paff (b,a) \vec \varphi, \paff (b,a)  \vec \phi} \vec{f}, \vec{g} \rangle  =
\langle  m W_{\varphi _1 \otimes \varphi _2} (f_1 \otimes f_2) ,
W_{\phi_1 \otimes \phi_2} (g_1 \otimes g_2)   \rangle,
\end{equation}
$ f_1,f_2,g_1,g_2 \in   L^{2}(\mathbb{R}^d) $,
can be interpreted as a bilinear extension of (two)\-wavelet localization operators considered in
\cite{Wong2002}.
More precisely, we have the following result, which seems to be  new (see also \cite[Theorem 19.11]{Wong2002}).

\begin{prop} \label{prop:waveletSp}
Let $ \h = L^{2}(\mathbb{R}^d) \otimes L^{2}(\mathbb{R}^d)$,
and let $ \varphi_1,$
$\varphi_2 $, $ \phi_1,$ and $ \phi_2 $ be admissible rotation invariant wavelets such that
$ \| \varphi_j \| = \| \phi_j \| = 1$, $ j =1,2$.
If $ m \in L^p (\mathbb{R}^{2d} \times  \mathbb{R}^2 \setminus \{ 0 \}, \frac{db da}{|a|^{d+1}}) $, $ 1\leq p<\infty,$ then
$ M_{m,  \paff (b,a) \vec \varphi, \paff (b,a)  \vec \phi} $ given by \eqref{eq:Mblopaffine}
belongs to Schatten class $ \mathcal{S}_p(\h).$
\end{prop}

\begin{proof}
By Lemma \ref{lm:WTtensprod} it follows that $F = \paff (b,a)   \vec \varphi$ and $G = \paff (b,a)   \vec \phi $ are
continuous  tight frames for $\h.$ Thus $M_{m, F, G} $ is a tensor product continuous frame multiplier, and by Theorem \ref{sec:schatten1}
it follows that $M_{m, F, G} \in \mathcal{S}_p(\h)$.
\end{proof}

\par

Finally, we combine STFT and wavelet continuous tight frames and consider bilinear localization operators of the ``mixed--form``.

\par

Consider the measurable space $(X,\mu) = (\mathbb{R}^{2d} \times (\mathbb{R}^{d} \times \mathbb{R}\setminus \{0\}), \mu )$
where $\mu$ is the product of $2d-$dimensional Lebesgue measure and the left Haar measure $ \frac{db da}{|a|^{d+1}}$.
If $\varphi \in L^2 (\mathbb{R}^{d})\setminus \{0\}$ and if $ \phi \in  L^2 (\mathbb{R}^{d}) $ is an admissible and rotation invariant wavelet, then
we define the STFT-Wavelet transform on $ \h = L^2 (\mathbb{R}^{d})\otimes  L^2 (\mathbb{R}^{d})$ as follows
\begin{equation} \label{eq:mixedtransform}
(V_\varphi \otimes W_\phi)(f_1 \otimes f_2)  =
\langle f_1, \pi(x,\omega) \varphi \rangle  \otimes \langle f_2, \paff (b,a)   \phi \rangle
\end{equation}

By orthogonality relations \eqref{eq:STFTortrel} and  \eqref{eq:WTortrel}, it follows that
$$
\langle (V_\varphi \otimes W_\phi)(f_1 \otimes f_2), (V_\varphi \otimes W_\phi)(g_1 \otimes g_2)\rangle_{L^2 (X)}
= \langle f_1 \otimes f_2,  g_1 \otimes g_2\rangle_{\h}  \|\varphi \|^2 \| \phi\|^2.
$$
Thus we conclude that
$ \pi(x,\omega) \varphi \paff (b,a)   \phi $
is continuous  tight frame for the tensor product space $\h$ (cf. Lemmas \ref{lm:STFTtensprod} and  \ref{lm:WTtensprod}).

If $m : X \rightarrow \mathbb{C} $ is a measurable function, then the
related tensor product continuous frame multiplier is given by
\begin{multline} \label{eq:Mblopmixed}
M_{m,   \pi(x,\omega) \varphi \paff (b,a) \phi, \pi(x,\omega) \varphi \paff (b,a) \phi} \vec{f}, \vec{g} \rangle \\
 =
\int_X  m (x)  (V_\varphi \otimes W_\phi)(f_1 \otimes f_2) (x),
(V_\varphi \otimes W_\phi) (g_1 \otimes g_2) (x) d\mu (x),
\end{multline}
for $ f_1,f_2,g_1,g_2 \in   L^{2}(\mathbb{R}^d) $.

In the same way as Propositions \ref{prop:waveletSp} and \ref{prop:STFTSp} we obtain the following.

\begin{prop} \label{prop:waveletSp2}
Let $ \h = L^{2}(\mathbb{R}^d) \otimes L^{2}(\mathbb{R}^d)$, and
$(X,\mu) = (\mathbb{R}^{2d} \times (\mathbb{R}^{d} \times \mathbb{R}\setminus \{0\}), dx d\omega \frac{db da}{|a|^{d+1}} )$.
Moreover, let
$\varphi \in L^2 (\mathbb{R}^{d})\setminus \{0\}$ and let  $ \phi  \in L^2 (\mathbb{R}^{d}) $ be an admissible and rotation invariant wavelet, such that $ \| \varphi \| = \|  \phi \| = 1$.
If $ m \in L^p (X) $, and $F= \pi(x,\omega) \varphi \paff (b,a) \phi $, $ 1\leq p<\infty,$  then
$ M_{m,  F, F} $ given by \eqref{eq:Mblopmixed}
belongs to Schatten class $ \mathcal{S}_p(\h).$
\end{prop}

\begin{proof} The result is a consequence of Theorem \ref{sec:schatten1} and the fact that  $ F= \pi(x,\omega) \varphi \paff (b,a)   \phi $ is a continuous  tight frame for  $\h$.
\end{proof}

\par

We refer to \cite{xxlgrospeck19} where a general approach based on the coorbit space theory is used to obtain
deep continuity results for related kernel type operators.

\section{Localization operators as density operators of  quantum systems} \label{sec:quantum}

In this section we first briefly recall the notion of a density operator or a density matrix
(as presented in e.g. \cite[Section 19]{Hall}), and then identify specific tensor product  continuous
frame  multipliers as density operators. This opens the possibility to use more general results from
Sections  \ref{sec:contframes} and
\ref{sec:multipliers} in the study of quantum systems.

\par

If $ \psi $ represents the wave function which describes the quantum system of
e.g. two spinless "distinguishable" particles moving in $\mathbb{R}^3 $, then
typically $ \psi = \psi (x,y) \in L^2 (\mathbb{R}^6)$, where $x$ is the position of the first particle,
and $y$ is the position of the second particle.
In general, there does not seem to be a way to associate a vector
$ \tilde \psi \in  L^2 (\mathbb{R}^3) $ which could sensibly describe the state of the first (or second) particle, see \cite{BB, Hall}. To overcome
this obstacle a more general notion of the "state" of a quantum system is introduced by associating
expectation value of an observable on  $ L^2 (\mathbb{R}^3) $ with respect to the wave function $\psi $.
This turned out to be the notion of {\em density operator} or {\em density matrix},
 which is uniquely determined by a given family of expectation values.
A density operator on the Hilbert space $\h$ is simply a non-negative, self-adjoint operator
$ \rho \in  \mathcal{S}_1(\h)$ such that
$ \text{Tr}_{\h} (\rho) = 1$.

A class of density operators, called Toeplitz operators is recently studied in \cite{deGosson2020, deGosson2021}. They correspond to quantum states obtained from a fixed function by position--momentum translations. This approach is closely related to the STFT multipliers, and we complement the investigations from  \cite{deGosson2020} by considering the corresponding partial traces (reduced density operators).

By partial trace theorem (Theorem \ref{th:partialtrace}), a density operator of a subsystem can be related to partial trace of the density operator for the whole system.
This procedure may give a reasonable description of a subsystem of   a bipartite system given by the tensor product Hilbert space
$\h = \h_1 \otimes \h_2$.
In particular, if $\rho  \in  \mathcal{S}_1(\h_1 \otimes \h_2)$  is of the form $ \rho = \rho_1 \otimes \rho_2 $, then the corresponding density
operators for subsystems $\h_j$, $ j = 1,2,$ given by partial trace theorem
are exactly $\rho_j$, $ j = 1,2$, cf. \cite[Theorem 19.13]{Hall}. Then the state is said to be {\em a separable state}.
The opposite direction, i.e. the existence of a pure state $ \rho$
such that given  $\rho_j$, $ j = 1,2$, are its partial traces is considered in e.g. \cite{Kly}.
Recently, for given $\rho_j$, $ j = 1,2$,
necessary and sufficient conditions for the existence of $\rho$ with $ \text{supp} \rho \subset \mathcal{X} \subseteq \h $ such that
$\rho_j$, $ j = 1,2$, are its partial traces are given in \cite{FGZ}.
These investigations lead to interesting insights related to different types of operator convergence. For example, the weak convergence is not
preserved under the partial trace. We refer to \cite{FGZ} for details in that direction.

\par

It is known that characteristic functions of a certain region in phase space give rise to trace class localization operators and may serve to
extract
time-frequency features of a signal when restricted to that region, see \cite{da88}.
Thus, it seems plausible to identify convenient tensor product continuous frame multipliers as "localized versions" of density operators of
bipartite systems, and use their partial traces to study the features of a subsystem. Of course, to be appropriate candidate of a density operator,
a multipliers has to satisfy certain conditions.
For the convenience we call them {\em admissible} multipliers.

\par

\begin{defn}\label{defadmisscontframemult}
Let $\textbf{M}_{m,F,G}$ be a tensor product continuous Bessel (frame) multiplier of $F$
and $G$ with respect to the  symbol $m$. Then, $\textbf{M}_{m,F,G}$ is {\em admissible} if it is non-negative, self-adjoint trace class operator
such that
$$
\text{Tr}_{\h} (\textbf{M}_{m,F,G}) = 1.
$$
\end{defn}

\par

Therefore, any admissible tensor product continuous Bessel (frame) multiplier
is a density operator.

As noted in Section \ref{sec:multipliers}, if  $F$ is a continuous frame,
$m(x)\geq \delta > 0$ and $\|m\|_{\infty}<\infty$, then $\textbf{M}_{m,F,F}$  is positive, self-adjoint
and invertible. For a given $F$, the trace of $\textbf{M}_{m,F,F}$ depends on the symbol $m$, which can be designed in such a way to ensure that
$\textbf{M}_{m,F,F}$ is in fact an admissible multiplier.

To illustrate this idea we consider  particular case of STFT multipliers.

\begin{thm}\label{thm:example}
Let there be given $\varphi, \phi \in  L^2  (\mathbb{R}^{d}) $ such that $ \langle \varphi, \phi \rangle \neq 0$.
If  $m \in L^1 (\mathbb{R}^{2d}) \cap L^\infty  (\mathbb{R}^{2d}) $, then
$$
\text{{\em Tr}}_{\h} ( M_{m, \pi(x,\omega)   \varphi, \pi(x,\omega)  \phi}) = \langle \varphi, \phi \rangle \int_{\mathbb{R}^{2d}} m (x,\omega) dx d\omega,
\;\;\;
x,\omega \in \mathbb{R}^{d},
$$
where $ M_{m, \pi(x,\omega)   \varphi, \pi(x,\omega)  \phi} $ is weakly given by
$$
\langle M_{m, \pi(x,\omega)   \varphi, \pi(x,\omega)  \phi} f ,g \rangle
= \langle  m V_{\varphi} f, V_{\phi} g  \rangle, \;\;\; f,g \in L^2  (\mathbb{R}^{d}).
$$
\end{thm}

\begin{proof}
By Definition \ref{definitioncontframemult} and Lemma \ref{lm:STFTtensprod}
it follows that  $ M_{m, \pi(x,\omega)   \varphi, \pi(x,\omega)  \phi} $ is a tensor product continuous frame multiplier. Furthermore, since  $m \in L^1 (\mathbb{R}^{2d}) $ by Proposition \ref{prop:STFTSp} we have that
$ M_{m, \pi(x,\omega)   \varphi, \pi(x,\omega)  \phi} $ is a trace class operator.

The rest of the  proof is similar to the proof of \cite[Theorem 16.1]{Wong2002} which is formulated in terms of
irreducible and square-integrable representations of locally compact Hausdorff groups. We give it here for the sake of completeness.
Let $ (e_n)_{n \in \mathbb{N}} $ be an ONB in $ L^2  (\mathbb{R}^{d})$. Then, by Fubini's theorem, Parseval's equality,  and since $\pi(x,\omega)$ acts unitary on $
 L^2  (\mathbb{R}^{d}) $  we obtain
\begin{multline*}
\text{Tr}_{\h} ( M_{m, \pi(x,\omega)  \varphi, \pi(x,\omega)   \phi}) =
\sum_{n=1} ^\infty \langle M_{m, \pi(x,\omega)  \varphi, \pi(x,\omega)   \phi} e_n, e_n \rangle
\\
=
\sum_{n=1} ^\infty \int_{\mathbb{R}^{2d}} m (x,\omega)  \langle   e_n, \pi(x,\omega)  \varphi \rangle
\langle \pi(x,\omega)   \phi  , e_n \rangle dx d\omega
\\
=  \int_{\mathbb{R}^{2d}} m (x,\omega) \sum_{n=1} ^\infty \langle   e_n, \pi(x,\omega)  \varphi \rangle
\langle \pi(x,\omega)   \phi  , e_n \rangle dx d\omega
\\
=  \int_{\mathbb{R}^{2d}} m (x,\omega)  \langle    \pi(x,\omega)  \varphi ,  \pi(x,\omega)   \phi   \rangle dx d\omega
\\
= \langle \varphi, \phi \rangle \int_{\mathbb{R}^{2d}} m (x,\omega)  dx d\omega,
\end{multline*}
and the proof is finished.
\end{proof}

\begin{prop}\label{prop:example}
Let there be given $\varphi_j, \phi_j \in L^2 (\mathbb{R}^{d} \setminus \{ 0\} $, and let
 $\vec \varphi = \varphi_1 \otimes \varphi_2$, $\vec \phi = \phi_1 \otimes \phi_2$.
If $m \in L^1 (\mathbb{R}^{4d}) \cap L^\infty  (\mathbb{R}^{4d}) $ is chosen so that
\begin{equation} \label{eq:trace-symbol}
\int_{\mathbb{R}^{4d}} m (x,\omega)  dx d\omega = \frac{1}{ \langle \vec \varphi,\vec \phi \rangle},
\end{equation}
then
$\text{{\em Tr}}_{\h} ( M_{m, \pi(x,\omega) \vec \varphi, \pi(x,\omega)  \vec \phi}) = 1,$
where $ M_{m, \pi(x,\omega)  \vec \varphi, \pi(x,\omega)  \vec \phi}$ is
given by  \eqref{eq:Mblop}.

If, in addition $\vec \varphi = \vec \phi,$ and $m > 0$, then $ M_{m, \pi(x,\omega)  \vec \varphi, \pi(x,\omega)  \vec \varphi}$
is an admissible multiplier.
\end{prop}

\begin{proof}
To proof the first part, it is enough to consider the  extension of  Theorem \ref{thm:example} to tensor product Hilbert space $ \h = L^2 (\mathbb{R}^{d}) \otimes L^2 (\mathbb{R}^{d})$.

The second part follows from the fact that $ M_{m, \pi(x,\omega)  \vec \varphi, \pi(x,\omega)  \vec \varphi}$ is the frame operator of
$\sqrt{m} \vec \varphi$ and so it is positive, self-adjoint and invertible. Since by Theorem \ref{thm:example} and  \eqref{eq:trace-symbol}
$$
\text{Tr}_{\h} ( M_{m, \pi(x,\omega)  \varphi, \pi(x,\omega)   \phi}) =
 \langle \vec \varphi,\vec \phi \rangle \int_{\mathbb{R}^{4d}} m (x,\omega)  dx d\omega = \langle \vec \varphi,\vec \phi \rangle  \cdot
  \frac{1}{ \langle \vec \varphi,\vec \phi \rangle}= 1,
$$
it follows that $ M_{a, \pi(x,\omega)  \vec \varphi, \pi(x,\omega)  \vec \varphi}$ is an admissible multiplier.
\end{proof}

By Proposition \ref{prop:example} we have the following important conclusion, which can be interpreted as a description of a separable state of a bipartite quantum system. This also gives an affirmative partial
answer to the question of  de Gosson  \cite[Section 5]{deGosson2020} related to the restriction of
the structure of a density operator  to its partial traces.

\begin{thm} \label{thm:density}
Let $ \h = \h_1 \otimes \h_2 = L^2 (\mathbb{R}^{d}) \otimes  L^2 (\mathbb{R}^{d}) $,
$\vec \varphi = \varphi_1 \otimes \varphi_2 \in L^2 (\mathbb{R}^{2d}) \otimes  L^2 (\mathbb{R}^{2d}) \setminus \{ 0\} $,
and let $  m_j (x_j, \omega_j) $ $ \in L^1 (\mathbb{R}^{2d}) \cap L^\infty  (\mathbb{R}^{2d}) $ be positive functions such that
\begin{equation} \label{eq:trace-symbol=1}
\int_{\mathbb{R}^{2d}} m_j (x_j,\omega_j)  dx_j d\omega_j = \frac{1}{\| \varphi_j \|}, \qquad j =1,2.
\end{equation}
Put $m (x,\omega) =   m_1 (x_1, \omega_1) m_2 (x_2, \omega_2)$, and
$$ F= \pi(x,\omega)  \vec \varphi = \pi(x_1,\omega_1)  \varphi_1 \otimes \pi(x_2,\omega_2)  \varphi_2 .$$
Then the operator $ M_{m, F,F}$
given by  \eqref{eq:Mblop} is a density operator, and its partial trace
$T (\textbf{M}_{m, F, F} )$  with respect to $\h_j$ is the density operator
$  \textbf{M}_{m_j,\varphi_j,\varphi_j},$ $ j =1,2.$

\end{thm}

\begin{proof}
By Proposition \ref{prop:example} it follows that  $ M_{m, F, F}$ is an admissible multiplier, and therefore it is a density operator.

Next, by Corollary \ref{cor:partialtrace} it follows that
$$T (\textbf{M}_{m, \pi(x,\omega)  \vec \varphi, \pi(x,\omega)  \vec \varphi} )=
\textbf{M}_{m_1,\varphi_1,\varphi_1} \text{Tr} (\textbf{M}_{m_2,\varphi_2,\varphi_2}).$$
From the assumptions of the theorem it follows that
$\textbf{M}_{m_1,\varphi_1,\varphi_1}$ and $ \textbf{M}_{m_2,\varphi_2,\varphi_2}$ are both admissible multipliers,
so that
$$T (\textbf{M}_{m, \pi(x,\omega)  \vec \varphi, \pi(x,\omega)  \vec \varphi} )=
\textbf{M}_{m_1,\varphi_1,\varphi_1},$$
and it is a density operator. Similarly for  $ \textbf{M}_{m_2,\varphi_2,\varphi_2}$.
\end{proof}

In the same manner one can consider multipliers given by \eqref{eq:Mblopaffine} and \eqref{eq:Mblopmixed},
and use Propositions \ref{prop:waveletSp} and \ref{prop:waveletSp2} to obtain another types of density operators for which Theorem \ref{thm:density} holds as well.
These considerations can be further used in the study of different aspects of  bipartite quantum systems.

\vspace{3mm}

\textbf{Acknowledgments}:
This work is supported by projects A\-NA\-C\-RES and TIFREFUS,
MPNTR of Serbia Grant No. 451--03--9/2021--14/200125,
{\em "Localization in Phase space: theoretical, numerical and practical aspects"} No. 19.032/961--103/19
MNRVOID Republic of Srpska, and
the project P 34624 {\em "Localized, Fusion and Tensors of Frames"} (LoFT)
of the Austrian Science Fund (FWF). The first author thanks Nora Simovich for help with typing.

\end{document}